\documentclass[11pt, letterpaper, oneside]{article}
\usepackage[colorlinks=false]{hyperref}
\usepackage{latexsym, amscd,stmaryrd, mathtools}
\usepackage{amsthm, amsmath, amssymb, amsfonts}
\usepackage{tikz-cd}
\tikzset{%
    symbol/.style={%
        ,draw=none
        ,every to/.append style={%
            edge node={node [sloped, allow upside down, auto=false]{$#1$}}}
    }
}
\usepackage[utf8]{inputenc}
\usepackage[T1]{fontenc}
\usepackage[english]{babel}
\usepackage[makeroom]{cancel}
\usepackage{bbm}
\usepackage{enumitem}

\DeclareMathOperator{\CE}{\mathrm{CE}}
\newcommand{\Vol}{dx^{0\ldots n}}
\DeclareMathOperator{\rescaling}{\Theta}
\DeclareMathOperator{\id}{\mathbbm{1}}
\DeclareMathOperator{\action}{\circlearrowleft}

\theoremstyle{plain}
	\newtheorem{thm}{Theorem}[section]
	\newtheorem*{bigthm}{Main Theorem}
	\newtheorem{prop}[thm]{Proposition}
	\newtheorem{lemma}[thm]{Lemma}
	\newtheorem{cor}[thm]{Corollary}

\theoremstyle{definition}
	\newtheorem{defi}[thm]{Definition}

\theoremstyle{remark}
	\newtheorem{remark}[thm]{Remark}
	\newtheorem{example}[thm]{Example}

\title{Multisymplectic actions of compact Lie groups on spheres}
\date{}
\author{
	Antonio Michele MITI 
	\thanks{
	KU Leuven, Department of Mathematics, Celestijnenlaan 200B box 2400, BE-3001 Leuven, Belgium;
	\href{mailto:antoniomichele.miti@kuleuven.be}{antoniomichele.miti@kuleuven.be};
	}
	\thanks{Dipartimento di Matematica e Fisica  ``Niccol\`o Tartaglia", Universit\`a Cattolica del Sacro Cuore, Via dei Musei 41, 25121 Brescia, Italy.}
	~ and Leonid RYVKIN
	\thanks{IMJ-PRG, Université de Paris, Batiment Sophie Germain, 8 place Aur\'elie Nemours, 75013 Paris, France
	\href{mailto:leonid.ryvkin@imj-prg.fr}{Leonid.Ryvkin@imj-prg.fr};}
}

\begin{document}
%------------------------------------------------------------------------------------------------

\maketitle

\begin{abstract}
	We investigate the existence of homotopy comoment maps (comoments) for high-dimensional spheres seen as multisymplectic manifolds.
	Especially, we solve the existence problem for compact effective group actions on spheres and provide explicit constructions for such comoments in interesting particular cases.
\end{abstract} 

\noindent {\bf MSC-classification (2010):} 	53D05, 53D20, 54H15, 57S15, 57S25,  22C05, 22F30, 37K05. \\


\tableofcontents

%----------------------------------------------------------------------------------------------------------------------------------
\section*{Introduction}
%----------------------------------------------------------------------------------------------------------------------------------
\addcontentsline{toc}{section}{Introduction}

\emph{Multisymplectic structures} (also called \emph{``$n$-plectic''}) are a rather straightforward generalization of symplectic ones where closed non-degenerate $(n+1)$-forms replace $2$-forms.

Historically, the interest in multisymplectic manifolds, i.e. smooth manifolds equipped with an $n$-plectic structure,  has been motivated by
the need of understanding the geometrical foundations of first-order classical field theories.
The key point is that, just as one can associate a symplectic manifold to an ordinary classical mechanical system (e.g. a single point-like particle constrained to some manifold), it is possible to associate a multisymplectic manifold to any classical field system (e.g. a continuous medium like a filament or a fluid).
It is important to stress that mechanical systems are not the only source of inspiration for instances of this class of structures. For example, any oriented $n$-dimensional manifold can be considered $(n-1)$-plectic when equipped with a volume form and semisimple Lie groups have a natural interpretation as $2$-plectic manifolds.

As proposed by Rogers in \cite{Rogers2010} (see also \cite{Zambon2010}), this generalization can be expanded by introducing a higher analogue of the Poisson algebra of smooth functions (also known as ``observable algebra'')  to the multisymplectic case.
Namely, Rogers proved that the algebraic structure encoding the observables on a multisymplectic manifold is the one of an $L_{\infty}$-algebra, that is, a graded vector space endowed with skew-symmetric multilinear brackets satisfying the Jacobi identity up to coherent homotopies.

The latter concept allowed for a natural extension of the notion of moment map, called \emph{(homotopy) comoment map}, originally defined in \cite{Callies2016}, associated to an infinitesimal action of a Lie group on a manifold preserving the multisymplectic form. 
As this concept is particularly subtle and technical, at the moment there are only few meaningful examples worked out in full detail. 
We can cite, for instance, \cite{Callies2016} for a broad account and \cite{Miti2018} regarding the comoment pertaining to volume-preserving diffeomorphisms acting on oriented Riemannian manifolds.

In this work, we try to address this problem by giving new insights on multisymplectic actions of compact groups and thus deriving existence results and explicit constructions for comoment maps related to actions on spheres.

\iffalse\paragraph{Main result:}\fi

\begin{bigthm}\label{thm:mainresult} (Proposition \ref{prop:intransitive} and Theorem \ref{thm:surprise})
	Let $G$ be a compact Lie group acting multisymplectically and effectively on the $n$-dimensional sphere $S^n$ equipped with the standard volume form, then the action admits a comoment map if and only if $n$ is even or the action is not transitive. 
\end{bigthm}

\noindent {\bf Interesting particular cases of the Main Theorem:}
\begin{itemize}
	\item The action of $SO(n)$ on $S^n$ is not transitive, hence it admits a comoment for all $n$. 
	We shall give an explicit construction for such a comoment in Subsection \ref{subsecson} that extends the construction given in \cite{Callies2016} only up to the $5$-dimensional sphere.
		\item The action of $SO(n+1)$ on $S^n$ only admits a comoment for even $n$. For the cases where such a comoment exist, giving explicit formulas seems to be a non-trivial task. We give explicit formulas for the first two components $f_1$ and $f_2$ in terms of the standard basis of $\mathfrak {so}(n)$ in Subsection \ref{subsectra}, leaving an explicit description of the higher components as an open question.
		The core idea will be to focus on the particular cohomology of the acting group rather than working on the analytical problem of finding the primitives required for the construction of the components of a comoment. 
		\item For $SO(4)$ acting on $S^3$ no comoment exists. However, this problem can be fixed by centrally extending the Lie algebra $\mathfrak {so}(n)$ to a suitable $L_\infty$-algebra (cf. \cite{Callies2016,Mammadova2019}).
\end{itemize}

\paragraph{Outline of the paper:} 
In the first section we survey the theory of  comoments in multisymplectic geometry, as introduced in \cite{Callies2016}. We include proofs for some known results in order to achieve a complete and self-contained exposition. 
The main novelty in this section is an intrinsic proof of Theorem \ref{cpteqcom}, which does not depend on the choice of a model for equivariant cohomology. 

We then prove the Main Theorem for the non-transitive case in Section \ref{secintrans} and the transitive case in Section \ref{sectra}. 
In addition to proving the abstract theorem, we give explicit constructions for important classes of group actions and highlight interesting phenomena.

\paragraph{Conventions:}
Given any cochain complex $C=(C^\bullet,d)$ we denote by $Z^k(C)=ker(d^{(k)})$ the subgroup of cocycles and by $B^k(C)=d~C^{k-1}$ the subgroup of coboundaries. 
In the case of chain complexes we employ the same notation with lowered indices.

We denote with $\partial: \Lambda^\bullet \mathfrak{g} \to \Lambda^{\bullet-1} \mathfrak{g}$ the boundary operator of the Chevalley-Eilenberg complex of a Lie algebra $\mathfrak{g}$, which is given explicitly on homogeneous elements by the following equation:
\begin{equation} \label{eq:CE_boun}
\partial (\xi_1 \wedge \xi_2 \wedge \dots \wedge \xi_k) := 
\sum_{1\leq i< j \leq n} (-1)^{i+j}\, [\xi_i, \xi_j] \wedge \xi_1 \wedge \dots
\wedge {\hat \xi}_i \wedge \dots \wedge {\hat \xi}_j \wedge \dots \wedge \xi_n,
\end{equation}
where $\hat{}$ denotes deletion as usual and with $\partial_0 = 0$, for all $\xi_i \in \mathfrak{g}$.
\\
Dually, we define the Chevalley-Eilenberg differential as
$\delta_{\CE}\colon\Lambda^n \mathfrak{g}^\ast \to \Lambda^{n+1} \mathfrak{g}^\ast$ whose action on an element $\phi \in \Lambda^\bullet \mathfrak g^*$ is given by $\delta_{CE} \phi := \phi \circ \partial$.

We consider the contraction operator $\iota$ to be defined on multi-vector fields. Specifically, the contraction with the wedge product of $k$ vector fields $\xi_i$ is given as follows
$$ \iota(\xi_1 \wedge \dots \wedge \xi_k) = \iota_{\xi_k}\dots\iota_{\xi_1} .$$

\paragraph{Acknowledgements:}
The authors thank Benjamin B\"ohme, Camille Laurent-Gengoux, Mauro Spera, Tilmann Wurz\-ba\-cher and Marco Zambon for helpful and motivating conversations.
A.M.M. acknowledges partial support by the National Group for Algebraic and Geometric Structures, and their Applications (GNSAGA – INdAM)
and by FWO under the EOS project G0H4518N (Belgium).
L. R. was supported by the Ruhr University Research School PLUS, funded by Germany’s Excellence Initiative [DFG GSC 98/3].

%----------------------------------------------------------------------------
\section{Comoments on multisymplectic manifolds}\label{secgeom}
%----------------------------------------------------------------------------
In this section, we give a short introduction to the theory of comoments in multisymplectic geometry focusing on the geometric description of their cohomological obstructions.
Most results can be found in the literature  \cite{Rogers2010,Callies2016,Fregier2015,zbMATH06448534}
but we present some of the proofs for a more clear and self-contained exposition.
Our main contribution is an intrinsic proof of Theorem \ref{cpteqcom} which does not depend on a choice of model for the equivariant cohomology.

\iffalse\subsection{Multisymplectic Geometry and Homotopy comoment map}\fi
%
\begin{defi}
	A \emph{pre-multisymplectic manifold} of degree $k$ is a pair $(M,\omega)$, where $M$ is a smooth manifold and $\omega\in\Omega^k(M)$ is a closed differential form. The manifold is called \emph{multisymplectic} if $TM\to \Lambda^{k-1}T^*M, ~v\mapsto \iota_v\omega$ is an injective bundle map. For fixed degree $k$ of the form such manifolds are also called ``$(k{-}1)$-plectic''. 
\end{defi}
\begin{example}[Symplectic manifolds]
	A symplectic manifold is, by definition, a 1-plectic manifold.
\end{example} 
\begin{example}[Oriented manifolds] 
	Any $n$-dimensional manifold equipped with a volume form is an $(n{-}1)$-plectic manifold.
\end{example} 
\begin{example}[Multicontangent bundles] 
	Consider a smooth manifold $N$, the corresponding \emph{Multicotangent bundle} $\Lambda = \Lambda^n T^\ast N$ is naturally $n$-plectic.
	Indeed $\Lambda$ can be endowed with a canonical multisymplectic $(n+1)$-form $\omega \coloneqq d \theta$ obtained from a tautological potential $n$-form $\theta \in \Omega^n(\Lambda)$ given by:
		\begin{displaymath}
		\begin{split}
			\left[ \iota_{u_1 \wedge \dots \wedge u_n} \theta \right]_\eta 
			&= \iota_{(T \pi)_\ast u_1 \wedge \dots \wedge (T \pi)_\ast u_n} \eta \\
			&= \iota_{u_1 \wedge \dots \wedge u_n} \pi^\ast \eta 
			\qquad \qquad \forall \eta \in \Lambda \, , \: \forall u_i \in T_\eta \Lambda 
		\end{split}
		\end{displaymath}
		where $\pi:\Lambda \to N$ is the bundle projection and $T\pi: T\Lambda \to T N$ is the corresponding tangent map.
		This construction is the ``higher analogue'' of the canonical symplectic structure naturally defined on any cotangent bundle. 
		Note, however, that this is not yet the ``higher analogue'' of a \emph{phase space}, see \cite{Gotay1998a} for further details or \cite{Ryvkin2018} for a more recent review.
\end{example}

\vspace{1em}
\iffalse \subsubsection{Special classes of forms and vectors} \fi
Exactly as it happens in symplectic geometry, fixing a (pre-)$n$-plectic structure $\omega$ on $M$ 
provides a criterion for identifying special classes of vector fields and differential forms. 
We will make use of the following nomenclature:
\begin{defi}
		A vector field $X \in \mathfrak{X}(M)$ is called \emph{multisymplectic} if it preserves the pre-multisymplectic form $\omega$, i.e. $\mathcal{L}_X \omega = d\iota_X\omega = 0$.
		If $\iota_X \omega$ is also exact, the vector field is called \emph{Hamiltonian}. 
		Accordingly, we define the subspace of \emph{Hamiltonian forms} as follows		
		$$
			\Omega^{n{-}1}_{Ham}(M,\omega)=\left\{\alpha~|~\exists\, v_\alpha\in\mathfrak X(M) 
			~:~ d\alpha=-\iota_{v_\alpha}\omega \right\}\subset \Omega^{n{-}1}(M).
		$$
		When $\omega$ is non-degenerate, the Hamiltonian vector field $v_\alpha$ associated to a fixed Hamiltonian form $\alpha$ is unique.
\end{defi}
Generalizing the construction of a Lie bracket on the observables (i.e. smooth functions, i.e. differential forms of degree zero) on a symplectic manifold,
we can construct a family of (multi)-brackets on Hamiltonian forms simply contracting $\omega$ with their corresponding Hamiltonian vector fields.
These brackets satisfy the Jacobi identity (and its higher analogues) up to total divergences, hence giving rise to a Lie algebra structure up to homotopy, i.e. an $L_\infty$-algebra.
\begin{defi}[\cite{Rogers2010}, cf. also \cite{MR1608547}]%[$L_\infty$-algebra of observables]
	The \emph{$L_\infty$-algebra of observables} $L_\infty(M,\omega)$ 
	of the (pre)-$n$-plectic manifold $(M,\omega)$ consists of a chain complex $L_\bullet$
	\begin{center}
		\begin{tikzcd}[column sep= small,row sep=small]
				0 \ar[r]& L_{n-1}\ar[symbol=\coloneqq,d] \ar[r]&
		\cdots\ar[r]&L_{k-2}\ar[symbol=\coloneqq,d]\ar[r]&\cdots\ar[r]&
		L_1\ar[symbol=\coloneqq,d]\ar[r]&L_0\ar[symbol=\coloneqq,d]\ar[r]&0\\
		& \Omega^0\ar["d",r]&\cdots\ar["d",r]&{\Omega^{n+1-k}}
		\ar["d",r]&\cdots\ar["d",r]&\Omega^{n-2}\ar["d",r]
		&\Omega^{n-1}_{\textrm{Ham}}&
		\end{tikzcd},
	\end{center}				
	which is a truncation of the de-Rham complex with inverted grading,
	endowed with $n$ (skew-symmetric) multibrackets $(2 \leq k \leq n+1)$
	\begin{equation}
		\begin{tikzcd}[column sep= small,row sep=0ex]
				[\cdot,\dots,\cdot]_k \colon& \Lambda^k\left(\Omega^{n-1}_{\textrm{Ham}}\right) 	\arrow[r]& 				\Omega^{n+1-k} \\
				& \sigma_1\wedge\dots\wedge\sigma_k 	\ar[r, mapsto]& 	\varsigma(k)\iota_{v_{\sigma_k}}\dots\iota_{v_{\sigma_1}}\omega 
		\end{tikzcd}		
	\end{equation}
	where $v_{\sigma_k}$ is any Hamiltonian vector field associated to $\sigma_k\in \Omega^{n-1}_{\textrm{Ham}}$ and $\varsigma(k) := - (-1)^{\frac{k(k+1)}{2}}$ is the Koszul sign.
\end{defi} 

\begin{lemma}[Theorem 5.2 in \cite{Rogers2010}] The $L_\infty$-algebra of observables $L_\infty(M,\omega)$ is an $L_\infty$-algebra.
\end{lemma}
\begin{proof} We refer to \cite{Rogers2010} for the original construction and the proof and to \cite{Ryvkin2016a} for a more elementary exposition including an introduction to $L_\infty$-algebras.
\end{proof}

\vspace{1em}
\iffalse \subsubsection{Multisymplectic actions and comoments} \fi
When one fixes a form $\omega$ on a manifold $M$ it is natural to highlight the group actions preserving this extra structure, also known as ``symmetries''.

\begin{defi}
A right action $\vartheta$ of a Lie group $G$ on $M$ is called \emph{multisymplectic} if it preserves the multisymplectic form, i.e. $\vartheta^*_g\omega=\omega$ for all $g\in G$, where $\vartheta_g=\vartheta(\cdot,g)$.
\\
We call \emph{(infinitesimal) multisymplectic Lie algebra action}
of $\mathfrak g$ on $M$ a Lie algebra homomorphism $\mathfrak g\to \mathfrak X(M), x\mapsto v_x$ such that $\mathcal{L}_{v_x}\omega=0$ for all $x\in\mathfrak g$.
\end{defi}
\begin{remark}
	In what follows, all the group actions considered are on the right.
	However, we decide to denote the action map as $\vartheta: G\times M \to M$, in place of the more natural choice  $\vartheta: G\times M \to M$, in order to agree with the sign conventions usually found in the literature.
\end{remark}
\begin{remark}
	For a connected Lie group $G$, a right action $\vartheta$ is multisymplectic if and only if the corresponding infinitesimal right action $v: \mathfrak{g}\to \mathfrak{X}(M)$
	given by  
	\begin{displaymath}
		v_\xi(m) = \left.\dfrac{d}{dt}\right\vert_0	\vartheta(m,\exp(t\xi)) \qquad \forall m \in M , \xi \in \mathfrak{g}
	\end{displaymath}		
	is multisymplectic.
	Note that, when considering left actions, the corresponding infinitesimal action is a Lie algebra anti-homomorphism.
\end{remark}

Let us now focus on the non degenerate, i.e. multisymplectic, case. 
In this context it is possible to define a higher analogue of the comoment map well-known in symplectic geometry:
\begin{defi}[\cite{Callies2016}]
	Let $v:\mathfrak g\to \mathfrak X(M)$ be a multisymplectic Lie algebra action.
	A \emph{homotopy comoment map} (or \emph{comoment} for short) pertaing to $v$ is 
	an $L_\infty$-morphism $(f)=\{f_i\}_{i=1,...,n}$ from $\mathfrak g$ to $L_\infty(M,\omega)$ satisfying $$df_1(\xi)=-\iota_{v_\xi}\omega \qquad \forall \xi\in\mathfrak{g}.$$ 
	A comoment for a Lie group action $\phi:M\times G\to M$ is defined as a comoment of its corresponding infinitesimal action. 
\end{defi}
More conceptually, a comoment is an $L_\infty$-morphism $(f):\mathfrak{g}\to L_\infty(M,\omega)$ lifting the action $v:\mathfrak{g}\to \mathfrak{X}(M)$, 
i.e. making the following diagram commute in the $L_\infty$-algebras category:
\begin{center}
\begin{tikzcd}[column sep = large]
	& L_\infty (M,\omega)\ar[d,"\pi"]\\
	\mathfrak{g} \ar[ur,"(f)",dashed]\ar[r,"v"'] & \mathfrak{X}(M)
\end{tikzcd}
\end{center}
where $\pi$ is the trivial $L_\infty$-extension\footnote{Note that any Lie algebra can be seen as an $L_\infty$-algebra concentrated in degree $0$, therefore any $L_\infty$-morphism $L\to\mathfrak{g}$ is simply given by a linear map $L_0 \to \mathfrak{g}$ preserving the binary brackets.}
of the function mapping any Hamiltonian form to the unique corresponding Hamiltonian vector field.
\\
In the following we will make use of an explicit version of this definition which is expressed by the following lemma:
\begin{lemma}[\cite{Callies2016}]
	A comoment $(f)$ for the infinitesimal	multisymplectic action	of ${\mathfrak g}$ on $M$ 
	is given explicitly by a sequence of linear maps
	\begin{displaymath}
		(f)  = \big\{ f_i: \,\,\, \Lambda^i {\mathfrak g} \to \Omega^{n-i}(M)
		\quad \vert \quad 
		0\leq i \leq n+1  \big\}	
	\end{displaymath}
	fulfilling a set of equations:
	\begin{equation}\label{eq:fk_hcmm}
	-f_{k-1} (\partial p) = d f_k (p) + \varsigma(k) \iota(v_p) \omega
	\end{equation}
	together with the condition
	\begin{displaymath}
		f_0 = f_{n+1} = 0 
	\end{displaymath}
	for all $p \in \Lambda^k\mathfrak{g}$ and $k=1,\dots n+1$.
	Here $\partial$ is the Chevalley-Eilenberg boundary operator defined in equation \eqref{eq:CE_boun}.
\end{lemma}

\begin{defi} A comoment $(f)$ is called {\it G-equivariant}, if
	\begin{displaymath}
	{\mathcal L}_{v_\xi} f_k({\mathbf b}) = f_1 ([\xi, {\mathbf b}])
	\qquad \forall \xi \in \mathfrak{g} \; , \mathbf{b} \in \Lambda^k \mathfrak{g},
	\end{displaymath}
	where $[\xi, {\mathbf b}]$ is the adjoint action of $\mathfrak{g}$ on $\Lambda^k \mathfrak{g}$ which, on decomposable elements, is given by the formula 
	\begin{equation}\label{eq:adjointactionwedge}
	[\xi, x_1\wedge\dots\wedge x_k] =
	\sum_{l=0}^k
	(-1)^{k-l} [v,x_l] \wedge x_k \wedge \dots \wedge \hat{x_l} \wedge \dots \wedge 	x_1.
	\end{equation}
\end{defi}

\subsection{Cohomological obstructions to comoment maps}
Consider an infinitesimal action of $\mathfrak{g}$ on the pre-$n$-plectic manifold $(M,\omega)$ preserving the form $\omega$.
As shown in \cite{Fregier2015, zbMATH06448534} a comoment for this action can be interpreted as a primitive of a certain cocycle in a cochain complex. 

\begin{defi}\label{def:comomentbicomplex} 
The bi-complex naturally associated to the action of $\mathfrak{g}$ on $M$ is defined by
	\begin{displaymath}
		(C_\mathfrak{g}^{\bullet,\bullet} = \Lambda^{\geq 1} 
		\mathfrak{g}^*\otimes \Omega^\bullet(M), \delta_\text{CE},d),	
	\end{displaymath}
	where $d$ denotes the de Rham differential and $\delta_{CE}:\Lambda^k\mathfrak g^*\to \Lambda^{k+1}\mathfrak g^*$ the Lie algebra cohomology differential, defined on generators by
	\begin{displaymath}
	\delta_{CE}(f)(\xi_1,...,\xi_k)=\sum_{i<j}(-1)^{i+j}f([\xi_i,\xi_j],\xi_1,\dots,\hat{\xi}_i,\dots,\hat{\xi}_j,\dots,\xi_{k}).
	\end{displaymath}
	The  corresponding total complex is given by
	\begin{displaymath}
		(C_\mathfrak{g}^{\bullet}, d_\text{tot} = 
		\delta_\text{CE}\otimes \text{id} + \text{id}\otimes d),
	\end{displaymath}
where, according to the Koszul sign convention, $d_{\text{tot}}$ acts on an element of $\Lambda^k \mathfrak{g}^*\otimes \Omega^\bullet(M)$ as $\delta_\text{CE} + (-1)^k d$.
\end{defi}
	
\begin{thm}[Proposition 2.5 in \cite{Fregier2015}, Lemma 3.3 in \cite{zbMATH06448534}]
Let $(M,\omega)$ be a pre-$n$-plectic manifold and $v:\mathfrak g\to \mathfrak X(M)$ be an infinitesimal multisymplectic action. 
The primitives of the natural cocycle
	\begin{equation}\label{eq:omegatildeobstruction}
		\tilde{\omega} = \sum_{k=1}^{n+1} (-1)^{k-1} \iota^k_\mathfrak{g} \omega \in C_\mathfrak{g}^{n+1},
	\end{equation}
where 
	\begin{align*}
		\iota^k_\mathfrak{g} \colon \Omega^\bullet(M)
		&\to \Lambda^k \mathfrak{g}^\ast \otimes \Omega^{\bullet-k}(M)
		\\ \omega&\mapsto \omega_k = 
		\left(p \mapsto \iota_{v_p} \omega  \right) ,
	\end{align*}
\noindent are in one-to-one correspondence with comoments of $v$.
In particular, a comoment exists if and only if~$[\tilde{\omega}]=0\in H^{n+1}(C_\mathfrak g^\bullet,d_ {tot})$.
\end{thm}

\begin{proof}  
Being linear maps, the components $f_k$ can be regarded as elements of a vector space
	\begin{displaymath}
		f_k \in \Lambda^k \mathfrak{g}^\ast \otimes \Omega^{n-k}(M)\cong \text{Hom}_{\text{Vect}}(\Lambda^k \mathfrak{g}, \Omega^{n-k}(M))
	\end{displaymath}
	satisfying equation (\ref{eq:fk_hcmm})
	or, equivalently, as vectors $\tilde{f}_k =\varsigma(k) f_k$ satisfying
	\begin{equation}\label{eq:fk_hcmm_tilde}
		\tilde{f}_{k-1}(\partial p ) + (-1)^k d \tilde{f}_k ( p) = (-1)^{k-1}\iota(v_p) \omega .
	\end{equation}
	The last equation is obtained multiplying Equation (\ref{eq:fk_hcmm}) by the sign factor 
	$\varsigma{(k-1)}$ and noting that $\varsigma{(k-1)}\varsigma{(k)} = (-1)^k$.
	\\
	Upon considering the direct sum of these vectors
	\begin{displaymath}
		\tilde{f} = \left(\sum_{k=1}^n \tilde{f}_k \right) \in 
		\bigoplus_{k=1}^n \left(\Lambda^k \mathfrak{g}^\ast \otimes \Omega^{n-k}(M)\right)
	\end{displaymath}
	equation (\ref{eq:fk_hcmm_tilde}) can be recast as:
	\begin{equation}\label{eq:fk_hcmm_tilde_complex_1}
		\left[\delta_{\text{CE}}\otimes \text{id} + \text{id}\otimes d \right] \tilde f =
		\sum_{k=1}^{n+1} (-1)^{k-1} \iota^k_\mathfrak{g} \omega
	\end{equation}
	where $\iota^k_\mathfrak{g}$ is the operator defined above.
	Note that we are implicitly using the Koszul convention, therefore the action of $\text{id}\otimes d$ on a homogeneous element $f_k \in \Lambda^k \mathfrak{g}^*\otimes \Omega^\bullet(M)$ yields $(-1)^k (\text{id}\otimes d) f_k$.
	\\
	If we set
		\begin{displaymath}
			\tilde{\omega} = \sum_{k=1}^{n+1} (-1)^{k-1} \iota^k_\mathfrak{g} \omega \in C_\mathfrak{g}^{n+1},
		\end{displaymath}
	equation (\ref{eq:fk_hcmm_tilde_complex_1}) corresponds to
	\begin{equation}\label{eq:fk_hcmm_tilde_complex_2}
		d_{\text{tot}} \tilde{f} = \tilde{\omega}
	\end{equation}
	which is exactly the condition of $\tilde{f}$ being a primitive of $\tilde{\omega}$.
	\\
	It follows from Lemma \ref{lemma:multicartan} that $d_{tot}\tilde \omega=0$ for all actions preserving $\omega$, therefore the vanishing of the cohomology class 
	$[\tilde{\omega}] \in H^{n+1}(C_\mathfrak{g}^\bullet)$
	is a necessary and sufficient condition for the existence of a comoment for the infinitesimal action of $\mathfrak{g}$ on $M$.
\end{proof}
	
\begin{remark}\label{rk:kuenneth}
By the K\"{u}nneth theorem, the cohomology $H^\bullet(C_\mathfrak{g}^\bullet)$ is isomorphic to $H^{\geq 1}(\mathfrak g)\otimes H^\bullet(M)$, we will give a geometric interpretation to this fact in the next section.
\end{remark}	

\begin{remark}\label{rk:cp_obsruction}
It is customary in part of the literature (see for example \cite{Callies2016}, \cite{Miti2018}, \cite{Mammadova2019}) to consider, as a cohomological obstruction to the existence of a comoment for $v:\mathfrak{g}\to \mathfrak{X}(M,\omega)$, the following cocycle in the Chevalley-Eilenberg complex of $\mathfrak{g}$
\begin{center}
\begin{tikzcd}[column sep= small,row sep=0ex]
	c^{\mathfrak{g}}_{p}=(\iota^{n+1}_\mathfrak{g}\omega)\big\vert_p :
		& \Lambda^{n+1} \mathfrak{g} \arrow[r] & \mathbb{R}
	\\
		& x_{1} \wedge \dots \wedge x_{n+1} 	\ar[r, mapsto]
		& (\iota( v_{1} \wedge \dots \wedge v_{n+1})\omega)\big\vert_p
\end{tikzcd}		
\end{center}
where $p\in M$ is a fixed point in $M$,
Lemma \ref{lemma:multicartan} guarantees that $\delta_\text{CE} c^{\mathfrak{g}}_{p} = 0$.
Note that, when $M$ is connected, the cohomology class $[c_p^{\mathfrak{g}}] \in H^{n+1}(\mathfrak{g})$ is independent of the point $p$.
\\
The vanishing of $[\tilde{\omega}]$ implies in particular that
$(\iota^{n+1}_\mathfrak{g}\omega) \in C^{n+1}_\mathfrak{g}$ must be a boundary, hence the vanishing of $[c_p^{\mathfrak{g}}]$ is a necessary condition for the existence of a comoment.
\\
Moreover, it follows from Remark \ref{rk:kuenneth} that when $H^{i}_{\mathrm{dR}}(M) =0$ for $1 \leq i \leq n-1$ the vanishing of $[c_p^{\mathfrak{g}}]$ is also a sufficient condition for the existence of a comoment.
\end{remark}

\subsection{A geometric interpretation of the obstruction class}\label{subsec:geomint}
In symplectic geometry  the existence of a comoment implies that $\omega$ can be extended to a cocycle in equivariant de Rham cohomology  (\cite{MR721448}). Following \cite{Callies2016}, we illustrate this fact by giving a geometric interpretation to the obstruction class $[\tilde \omega]$ defined above and explain its analogue in the multisymplectic setting.\\

When the Lie algebra action $v$ comes from a Lie group action, we can interpret the complex $C^\bullet_\mathfrak g$ and the cocycle $\tilde \omega$ in terms of the de Rham cohomology of invariant forms.
\begin{defi}
Let $\vartheta:G\times M\to M$ be a Lie group action. We denote by $\Omega^\bullet(M,\vartheta)$ the subcomplex of $\vartheta$-invariant differential forms. The cohomology of this complex is called \emph{invariant de Rham cohomology} and denoted by $H^\bullet(M,\vartheta)$.
\end{defi}
\begin{remark}
It is more common in the literature to denote these invariant spaces by $\Omega^\bullet(M)^G$ and $H^\bullet(M)^G$. We use the above notation to emphasize the specific Lie group action involved.
\end{remark}
\begin{remark}\label{comactintegration}
The invariant cohomology is not the same as the equivariant cohomology, which we will define later. For example, whenever $G$ is compact, the natural map $H^\bullet(M,\vartheta)\to H^\bullet(M)$ induced by the inclusion of the subcomplex is an isomorphism, as in this case any form can be made invariant by averaging. Pullbacks along equivariant maps lead to homomorphisms of the invariant cohomology groups. For details we refer to \cite{MR0336651}.
\end{remark}
\begin{lemma}[Lemma 6.3 in \cite{Callies2016}]\label{old_infmom}
Let $\vartheta:G\times M\to M$ be a right Lie group action. 
We denote by $(r\times id):G\times (G\times M)\to (G\times M), (h,(g,m))\mapsto (gh,m) $ the right multiplication action on the second factor. The complex ~$\Omega^\bullet(G\times M,{r\times id})$ is naturally isomorphic to
 ~$C_\mathfrak g^\bullet\oplus( \Lambda^0\mathfrak g^*\otimes \Omega^\bullet(M))$. 
\end{lemma}

\begin{proof}  We have a natural map 
	\begin{equation}\label{eq:naturalmap}
	 \Lambda^\bullet\mathfrak g^* \otimes \Omega^\bullet(M)\to\Omega^\bullet(G,r) \otimes \Omega^\bullet(M)\to \Omega^\bullet(G\times M, {r\times id}).
	\end{equation}
	The first arrow comes from the isomorphism $\Lambda^k \mathfrak{g}^\ast \to \Omega^k(G,r)$, which associates to any element of $\Lambda^k\mathfrak g^*=\Lambda^kT^*_eG$ its right-invariant extension.
	The second arrow comes from the exterior wedge product i.e. from the map
	\begin{displaymath}
		\begin{tikzcd}[column sep= small,row sep=0ex]
				\Omega^q(G,r) \otimes \Omega^p(M) 	\arrow[r]
				& \Omega^{q+p}(G\times M, {r\times id})
				\\
				\alpha\otimes\beta 	\ar[r, mapsto]
				&	\pi_1^*\alpha \wedge \pi_2^* \beta
		\end{tikzcd}		
	\end{displaymath}
	where $\pi_i$ are the projections on the $i$-th factor of $G\times M$.
	Regarding the complexes involved as graded vector spaces, the previous map can be extended to a degree 0, bilinear map
	\begin{equation}\label{eq:kunnetqiso}
		\begin{tikzcd}[column sep= small,row sep=0ex]
				\Omega^\bullet(G,r) \otimes \Omega^\bullet(M) 	\arrow[r]
				& \Omega^{\bullet}(G\times M, {r\times id})
				\\
				\alpha\otimes\beta 	\ar[r, mapsto]
				&	(-1)^{|\alpha|}\pi_1^*\alpha \wedge \pi_2^* \beta
		\end{tikzcd}		
	\end{equation}
	where the extra signs comes from the Koszul convention we employed in definition \ref{def:comomentbicomplex} when defining the differential in the total complex.
	The map \eqref{eq:naturalmap} admits an inverse which takes $\alpha\in \Omega^\bullet(G\times M, r\times id)$ to $\alpha|_{\{e\}\times M}\in \Gamma(\Lambda^\bullet \mathfrak g^*\otimes T^*M)=\Lambda^\bullet\mathfrak g^*\otimes\Omega^\bullet(M)$.
	\\
	The statement follows from the observation that $C_\mathfrak g^\bullet\oplus( \Lambda^0\mathfrak g^*\otimes \Omega^\bullet(M))$ is the total complex of $\Lambda^\bullet\mathfrak g^*\otimes \Omega^\bullet(M)$ and that the second arrow defined above is precisely the function inducing the K\"unneth isomorphism \cite{MR658304}.
\end{proof}

\begin{prop}[Proposition 6.4 in \cite{Callies2016}]\label{infmom}
Assume that $G$ preserves a pre-mul\-ti\-sym\-plec\-tic form $\omega$. 
Let $v$ be the infinitesimal action induced by $G$.
Then the cocycle $\tilde \omega \in C_\mathfrak g^{n+1}=\Omega^\bullet(G\times M, {r\times id})$ with respect to the infinitesimal action $v$ of $\mathfrak{g}$ induced by $\vartheta$, is given by $\tilde{\omega}=\vartheta^*\omega-\pi^*\omega$, where $\pi:G\times M\to M$ is the projection onto the second factor.
\end{prop}
\begin{proof}
	Being an action, the map $\vartheta\colon G \times M \to M$ is equivariant (with respect to $r\times id$ in the domain and $\vartheta$ in the target). Hence, the cochain-map $\vartheta^*:\Omega^\bullet(M)\to \Omega^\bullet(G\times M)$ 
	restricts to a well-defined map on the invariant subcomplexes 
	$\vartheta^*:\Omega^\bullet(M,{\vartheta})\to \Omega^\bullet(G\times M,{r \times id}) $,
	and in particular we have a well-defined map in cohomology.
	\\
	Let $X_1,...,X_{n+1-i}\in T_pM$ and $\xi_1,...,\xi_{i}\in T_eG=\mathfrak g$.
	For all $ 1<i\leq n+1 $ we get
	\begin{align*}
		\vartheta^*\omega(\xi_1,...,\xi_{i},X_1,...,X_{n+1-i}) & = \omega(\vartheta_*\xi_1,...,\vartheta_*\xi_{i},\vartheta_*X_1,...,\vartheta_*X_{n+1-i}) =\\
		&=\omega(v(\xi_1),...,v(\xi_{i}),X_1,...,X_{n+1-i})=\\
		&= (\iota^i_\mathfrak g\omega)(\xi_1,...,\xi_{i})(X_1,...,X_{n+1-i})
	\end{align*} 
	and for $i=0$ we get
	\begin{displaymath}
		\vartheta^*\omega(X_1,...,X_{n+1} )=
		\omega(X_1,...,X_{n+1}).
	\end{displaymath}		

This means $\vartheta^*\omega= -\tilde\omega + 1\otimes \omega$, where the first summand is the image of $\tilde{\omega}$ under the map defined in equation \eqref{eq:kunnetqiso} and the last summand comes from the case $i=0$.  
The observation that $1\otimes \omega=\pi^*\omega$, for $\pi:G\times M\to M$ the natural projection, finishes the proof.  One should note, that this is true although $\pi$ is not an equivariant map with respect to $( r\times id, \vartheta)$.
\end{proof}

\begin{remark}We think that the above proposition is central to the understanding of multisymplectic comoments, as it enables an elementary and unified treatment of many phenomena in multisymplectic geometry.
\begin{itemize}
\item For symplectic manifolds, this result gives a nice interpretation for the result of \cite{MR721448} that moment maps are in correspondence to equivariant extensions of $\omega$ and also explains why this correspondence fails in the general multisymplectic setting (cf. Example \ref{exnongen}). 
\item Let $G_i$ act on the multisymplectic manifolds $(M_i,\omega_i)$ for $i\in\{1,2\}$. If there exist comoments for $G_1$ and $G_2$, then there exists a comoment for $G_1\times G_2$ on the multisymplectic manifold $(M_1\times M_2, \pi_1^*\omega_1\wedge \pi_2^*\omega_2)$. This theorem from \cite{Shahbazi2016} can be derived from Proposition \ref{infmom}.
\item A comoment exists, whenever the multisymplectic form $\omega$ can be lifted to a class in the equivariant cohomology $H_G^{n+1}(M)$.
(See Theorem \ref{cpteqcom}). 
\end{itemize} 
\end{remark}

\begin{remark}\label{rk:invpot}
Especially, Proposition \ref{infmom} immediately implies that a $\vartheta$-invariant potential of $\omega$ induces a comoment, 
as an invariant potential $\alpha$ of $\omega$ would be mapped to a potential $(\vartheta^*\alpha-\pi^*\alpha) \in \Omega^\bullet( G\times M , r\times id)$ of $\tilde{\omega}=\vartheta^*\omega-\pi^*\omega$.
Note that $\omega$ being exact is not a sufficient condition, 
because a primitive $\vartheta^*\alpha-\pi^*\alpha$ need not to be an element in the invariant cochain complex $\Omega^\bullet(G\times M, r\times id)$ in general. \\ 
\end{remark}

When such an invariant potential exists, it is fairly easy to give an explicit expression for the components of a comoment, as illustrated by the following Lemma:
\begin{lemma}[Section 8.1 in \cite{Callies2016}]\label{lem:extexact}
	Let $M$ be a manifold with a $G$-action, 
	let $\alpha\in \Omega^n(M,G)$ be a $G$-invariant $n$-form and consider the pre-$n$-plectic form $\omega=d\alpha$ on $M$.\\
	The action $G\action \left(M,d\alpha\right)$ admits a $G$-equivariant comoment map $(f):\mathfrak{g} \to L_{\infty}(M,\omega)$, given by $(k=1,\dots,n)$:
\begin{align*}
	f_{k} \colon \Lambda^k\mathfrak{g} &\to \Omega^{n-k}(M)\\
	q&\mapsto (-1)^{k-1}\varsigma(k)\iota(v_q)(\alpha) .
\end{align*} 
\end{lemma}
\begin{proof}
	A direct proof of this statement can be given by showing that equation \eqref{eq:fk_hcmm} is satisfied.
	Upon employing Lemma \ref{lemma:multicartan}, we have:
	\begin{displaymath}
	\begin{split}
		\textrm{d} f_m (p) &= (-1)^{m-1} \varsigma(m) \textrm{d} \iota_{v_p} \alpha =
		-\varsigma(m) (-1)^{m} \textrm{d} \iota(v_1\wedge\dots\wedge v_m) \alpha =\\
		&= -\varsigma(m) \left(\iota_{v_p} \textrm{d}\alpha + \iota_{\partial v_p} \alpha +
		\sum_{k=1}^{m} (-1)^k  \iota( x_1\wedge\, \hat{x_k}\, \wedge\dots\wedge x_m) \cancel{\mathcal{L}_{x_k} \alpha}
		 \right) =\\
		&= -\varsigma(m) \iota_{v_p} \omega +  (-1)^{m-1} \varsigma(m-1) \iota_{\partial v_p} \alpha  =
		-\varsigma(m) \iota_{v_p} \omega - f_{m-1}(\partial v_p) ,
	\end{split}
	\end{displaymath}	
	thus $G$-equivariance follows from Equation \ref{eq:multiliecartan}.
\end{proof}
\begin{cor}[$SO(n)$-action on $\mathbb{R}^n$, Example 8.4 in \cite{Callies2016}]\label{cor:sorn}
	The canonical action 
	$$SO(n) \action \left( \mathbb{R}^{n}, dx^{1\dots n}\right),$$
	where~$x = (x^i)$ are the standard coordinates on~$\mathbb{R}^n$ 
and~$dx^{1\dots n} = d x^1\wedge\dots \wedge dx^{n}$ is the standard volume form of $\mathbb{R}^n$, admits a comoment given by $(k=1,\dots,n)$:
\begin{align*}
	f_{k} \colon \Lambda^k\mathfrak{g} &\to \Omega^{n-1-k}(M)\\
	q&\mapsto (-1)^{k-1} \frac{\varsigma(k)}{n} \iota(E \wedge v_q)\left(dx^{1\dots n}\right)
\end{align*}
 	where $E =\sum_i x^i\partial_i$ is the Euler vector field.
\end{cor}
\begin{proof}
    The proof follows from Lemma \ref{lem:extexact} noting that the standard volume form admits the $G=SO(n)$ invariant form
    $$\alpha = \dfrac{\iota_E\left(dx^{1\dots n}\right)}{n}$$
    as a primitive.
\end{proof}

We will be primarily interested in the case of compact Lie groups. In this case,  we do not have to care about the invariance of forms:
\begin{cor}\label{cor:core}
Let $\vartheta:G\times M\to M$ be a compact Lie group acting on a pre-multisymplectic manifold, preserving the pre-multisymplectic form $\omega$. 
A comoment exists if and only if $[\vartheta^*\omega-\pi^*\omega]=0\in H^{n+1}(G\times M)$. 
\end{cor}
\begin{proof}
From Proposition \ref{infmom} we get the following sequence of maps between complexes together with the induced maps in cohomology:
\begin{center}
\begin{tikzcd}
 \Omega^\bullet(M,\vartheta) \ar[d,"\vartheta^\ast-\pi^\ast"] &\quad
 & H_\text{dR}(M) \ar[d,"\vartheta^\ast-\pi^\ast"]  
 & \lbrack \omega \rbrack \ar[d,mapsto]
 \\ 
 \Omega^\bullet(G\times M, r\times id) \ar["\cong",leftrightarrow]{d} &\quad
 & H_\text{dR}(G\times M) \ar[leftrightarrow,"\cong"]{d}[swap]{\text{\tiny (K\"unneth)}} 
 & \lbrack \vartheta^\ast \omega - \pi^\ast \omega \rbrack \ar[ddd,mapsto]
 \\ 
 \Omega^\bullet(G,r) \otimes \Omega^\bullet(M) \ar["\cong",leftrightarrow]{d}[swap]{} &\quad
 & H_\text{dR}(G) \otimes  H_\text{dR}(M) \ar[d,"\cong",leftrightarrow]
 \\ 
 \Lambda^\bullet \mathfrak{g}^* \otimes \Omega^\bullet(M)\ar["\cong",leftrightarrow]{d} &\quad
 & H_\text{CE}(\mathfrak{g}) \otimes  H_\text{dR}(M) 
 \ar["\cong",leftrightarrow]{d} & 
 \\
 C_\mathfrak{g}^\bullet \oplus ( \mathbb{R}\otimes \Omega^\bullet(M))&\quad 
 & H(C_\mathfrak{g}^\bullet)\oplus H_\text{dR}(M)
 & \lbrack \tilde{\omega}\rbrack
\end{tikzcd}
\end{center}
The statement follows by resorting to Remark \ref{comactintegration}, i.e. noting that 
$H_\text{dR}(G) \cong H(G,r)$ and $H_\text{dR}(G\times M) \cong H(G\times M, r\times id)$ via the averaging trick on compact Lie groups.
\end{proof}

\medskip
\noindent
We will now investigate the connection between comoments and equivariant cohomology.
\begin{defi}\label{def:equivariantcoho}
Let a compact Lie group $G$ act on a manifold $M$. Let $EG$ be a contractible space on which $G$ acts freely by $\vartheta^{EG}$. Then we define the equivariant cohomology of $M$ as $H^\bullet_G(M):=H^\bullet((M\times EG)/G)$, where $G$ acts on $M\times EG$ diagonally.
\end{defi}

\begin{remark}
As $EG$ might not be a manifold, we have to interpret $H^\bullet(\cdot)$ as a suituable cohomology theory (e.g. singular cohomology with real coefficients) in the above definition. 
\end{remark}

As $G$ is compact, when $\vartheta:G\times M\to M$ is a free action, we have $H_G^\bullet(M)=H^\bullet_{dR}(M/G)$. For a not necessarily free action $\vartheta$, we still have the following diagram
\begin{center}
\begin{tikzcd}[column sep=large]
	G\times (M\times EG)  \ar[r,shift left=1.5ex,"\vartheta\times \vartheta^{EG}"]\ar[r,shift right=1.5ex,,"\pi"]
	& M\times EG \ar[r,"q"]& (M\times EG)/G,
\end{tikzcd}
\end{center}
where $q$ is the projection to the orbits, which induces a sequence in cohomology 
(where we use the contractibility of $EG$ in the  left and middle term):
\begin{center}
\begin{tikzcd}[column sep=large]
	H^\bullet(G\times M) & H^\bullet(M) \ar[l,"\vartheta^*-\pi^*"]
	& \ar[l,"q^*"]H^\bullet_G(M)
\end{tikzcd}
\end{center}
Then $q\circ \vartheta=q\circ\pi$ implies $(\vartheta^*-\pi^*)\circ q^*=0$.
Using Remark \ref{comactintegration} and Proposition \ref{infmom} we get the following statement:
\begin{thm}[\cite{Callies2016}]\label{cpteqcom}
Let $G\times M\to M$ be a compact Lie group preserving a pre-multisymplectic form $\omega$. If $[\omega]\in H^\bullet(M)$ lies in the image of $q^*:H^\bullet_G(M)\to H^\bullet(M)$, then $\vartheta$ admits a comoment.
\end{thm}
\begin{remark}
The advantage of our approach to the theorem is that it is much simpler and more intrinsic: We do not need to choose a model for equivariant de Rham cohomology.
In Section 6 of \cite{Callies2016} can be found similar results framed in the Bott-Shulman-Stasheff and in the Cartan model.
 Furthermore, in Section 7.5 of \cite{Callies2016}, is discussed a generalization of this statement to non-compact groups.
 \end{remark}

Unfortunately, unlike the symplectic case (cf. \cite{MR721448}), the converse statement does not hold in general. 
Even if a (pre-)multisymplectic action of $G$ on $(M,\omega)$ admits a comoment, $[\omega]$ does not need to come from an equivariant cocycle. 
We will illustrate this fact by an example.

\begin{example} \label{exnongen}
Consider the action given by the Hopf fibration $\vartheta: S^1\times S^3\to S^3$. 
Let $\omega$ be the standard volume on $S^3$. By Remark $\ref{rk:kuenneth}$ the obstructions to a comoment lie in $H^1(\mathfrak u(1))\otimes H^2(S^3) $, $ H^2(\mathfrak u(1))\otimes H^1(S^3)$ and  $H^3(\mathfrak u(1))\otimes H^0(S^3) $, which are trivial. Hence, a comoment exists. 
\\
As the action is free (and the quotient is $S^2$), we have $H^3_{S^1}(S^3)=H^3(S^2)=0$. But $[\omega]\neq 0$ in $H^3(S^3)$, so the class $[\omega]$ cannot come from an equivariant cocycle.
\end{example}

\begin{remark}
We note that this example has a different character than the ones provided in Section 7.5 of \cite{Callies2016}. They exhibit cases, where individual comoments do not come from equivariant cocycles, whereas in our case no equivariant cocycle can be found for any of the possible comoments.
\end{remark}

%----------------------------------------------------------------------------
\section{Non-transitive multisymplectic group actions on spheres}\label{secintrans}
%----------------------------------------------------------------------------
The goal of this section is to prove the existence of comoments for non-transitive actions and construct an explicit comoment for the $SO(n)$-action on $S^n$.
\begin{lemma}\label{lem:core}
	Let $\vartheta: G\times S^n\to S^n$ be a compact Lie group acting multisymplectically on $S^n$ equipped with the standard volume form $\omega\in \Omega^n(S^n)$. Let $p\in S^n$ be any point. Then a comoment exists if and only if $\vartheta_p^*[\omega]\in H^{n}(G)$ is zero, where $\vartheta_p:G\to S^n$ is the map $g\mapsto gp$.
\end{lemma}
\begin{proof}
	By Corollary \ref{cor:core}, we only have to check that
	$$ 
	[\vartheta^*\omega-\pi^*\omega]=0\in H^{n}(G\times S^n)
	\quad\Leftrightarrow\quad 
	\vartheta_p^*[\omega]= 0\in H^{n}(G)~.$$
	
	The direct implication follows by considering the map $i:G\to G\times S^n,~~g \mapsto (g,p)$ and its induced linear map in cohomology $i^* : H^\bullet(G \times S^n) \to H^\bullet(S^n)$ which acts on $[\vartheta^*\omega - \pi^*\omega] \in H^n(G\times S^n)$ as
\begin{displaymath}
	i^* [\vartheta^* \omega - \pi^* \omega ] = 
	[(\vartheta \circ i)^\ast \omega - \cancel{(\pi \circ i)^\ast \omega}] = \vartheta^*_p[\omega] ,
\end{displaymath}
	because $\vartheta \circ i = \vartheta_p$ and $\pi \circ i $ is the constant map valued in $p\in S^n$.
	
	For the converse implication, note at first that the cohomology of the sphere implies
	$$H^{n}(G\times S^n)=(H^n(G)\otimes H^0(S^n))\oplus (H^0(G)\otimes H^n(S^n)) .	$$ 
	 Therefore, recalling Proposition \ref{infmom}, the obstruction $[\vartheta^*\omega-\pi^*\omega]$ lies entirely in $H^n(G)\otimes H^0(S^n)$ as $[\tilde{\omega}]$ has null component in $ H^0(G)\otimes H^n(S^n)$.
	 Since the restriction of $i^*$ to $H^n(G)\otimes H^0(S^n)$ is an isomorphism (the $0$-th cohomology group of a connected manifold is isomorphic to $\mathbb{R}$), one can conclude that $[\vartheta^*\omega-\pi^*\omega]$ vanishes if and only if 
	$$i^*[\vartheta^*\omega-\pi^*\omega]=\vartheta_p^*[\omega]=0\in H^n(G) .$$
\end{proof}
\begin{remark}
	Note that the direct implication in the proof of Lemma \ref{lem:core} does not depends from being the base manifold a sphere.	
	In other words, a necessary condition for the existence of a comoment is that the pullback of $\omega$ with respect to any orbit map vanishes.
\end{remark}
\begin{remark}
Observe that a result similar to Lemma \ref{lem:core} can be stated for any compact multisymplectic action $G \action (	M,\omega)$, with $\omega$ in degree $n+1$, such that the following cohomological condition holds
$$
\bigoplus_{k=1}^n H^k(G)\otimes H^{n-k}(M) = 0 .
$$
In particular, the same result is true for the action of any compact, connected and semisimple Lie group $G$ (i.e. such that $H^1(\mathfrak{g})=H^2(\mathfrak{g})=0$) acting on a $2$-plectic manifold.
See \cite[Proposition 7.1]{Callies2016} for an existence result for comoments related to this kind of actions in presence of a fixed point.
\end{remark}
\begin{prop}\label{prop:intransitive}
Let $G$ be a compact Lie group acting non-transitively on $S^n$ and preserving the standard volume form. Then $G$ admits a comoment.
\end{prop}
\begin{proof} If $G$ acts non-transitively then there exists an orbit $O\subset S^n$ of dimension strictly less than $n$. Let $p\in O$. Then we have $\vartheta_p^*[\omega]=\vartheta_p^*[\omega|_O]$, but $\omega|_O\in \Omega^n(O)$ is zero due to dimension reasons. Hence, the action admits a comoment, due to Lemma \ref{lem:core}.
\end{proof}
\begin{remark}
It is worth to mention that for non-transitive, free and proper actions it is possible to state the multisymplectic equivalent of the "symplectic reduction" procedure, see Example 2.2 in \cite{Bursztyn2019}.
\end{remark}

\subsection{Induced comoments and isotropy subgroups}
In this subsection we will show that comoments behave well under restriction to subgroups and invariant submanifolds. This will be useful for constructing an $SO(n)$-comoment for $S^n$ in the next subsection. 
Let $G$ be a Lie group with Lie algebra $\mathfrak{g}$, acting on a pre-$n$-plectic manifold $(M,\omega)$ with comoment map $(f)\colon \mathfrak{g}\to L_{\infty}(M,\omega)$. 
One can obtain new actions either restricting to a Lie subgroup of $G$ or restricting to an invariant submanifold of $(M,\omega)$. 
\begin{prop}[Lemma 3.1 in \cite{Shahbazi2016}]\label{prop:restrict}
	Let $H\subset G$ be a Lie subgroup, and denote by $j\colon \mathfrak{h}\hookrightarrow \mathfrak{g}$ the corresponding Lie algebra inclusion. 
	The restricted action of $H$ on $(M,\omega)$ has comoment $(f \circ j): \mathfrak{h}\to L_{\infty}(M,\omega)$, given by $f_i\circ j:\Lambda^i\mathfrak{h}\to \Omega^{n-i}(M)$ for $i=1,\dots,n$.
\end{prop}
\begin{center}
	\begin{tikzcd}
		H \ar[symbol=\action]{r} \ar[hook]{d}& M \ar[equal]{d}& & \mathfrak{h} \ar[hook,"j"']{d} \ar[dashed]{dr} & \\
		G \ar[symbol=\action]{r}& M & & \mathfrak{g}\ar["(f)"']{r} & L_{\infty}(M,\omega) 
	\end{tikzcd}
\end{center}
\begin{prop}[Lemma 3.2 in \cite{Shahbazi2016}]
	\label{prop:NenMGinvariant} 
	Let $N\overset{i}{\hookrightarrow} M$ be a $G$-invariant submanifold of $M$.
	\\
	Then the action $G\action \left(N,i^{\ast}\omega\right)$ has comoment $(i^{\ast} \circ f): \mathfrak{g}\to L_{\infty}\left(N,i^{\ast}\omega\right)$
	, given by $i^*\circ f_i:\Lambda^i\mathfrak{h}\to \Omega^{n-i}(N)$ for $i=1,\dots,n$.
\end{prop}
\begin{center}
	\begin{tikzcd}
		G \ar[symbol=\action]{r} \ar[equal]{d}& N \ar[hook,"i"]{d}& & & L_{\infty}(N, i^\ast\omega) \\
		G \ar[symbol=\action]{r}& M & & \mathfrak{g}\ar["(f)"']{r}\ar[dashed]{ur} & L_{\infty}(M, \omega) \ar["i^\ast"']{u} 
	\end{tikzcd}
\end{center}

Also, we can produce a new comoment by considering a different multisymplectic form obtained contracting $\omega$ with cycles in Lie algebra homology
\begin{equation}
	Z_k(\mathfrak{g}) = \{ p \in \Lambda^k \mathfrak{g} \; \vert: \partial p = 0 \}
	\quad.
\end{equation}
\begin{prop}[Proposition 3.8 in \cite{Ryvkin2016}]\label{prop:indmomap}
	Let $p\in Z_{k}(\mathfrak{g})$, for some $k \geq 1$,
	denote by $G_{p}$ the corresponding isotropy group for the adjoint action of $G$ on $\Lambda^{k}\mathfrak{g}$, and by $\mathfrak{g}_p=\{x\in \mathfrak{g}: [x,p]=0\}$ its Lie algebra.
	\\
	If $G_p^0$ is the connected component of the identity in $G_p$, then the action $G_p^0\action \left(M,\iota(v_p)\omega\right)$ admits a comoment
	$(f^p) \colon \mathfrak{g}_p \to L_{\infty}(M,\iota(v_p)\omega)$ with components $(i=1,\dots,n-k)$:
	\begin{align*}
	f^p_i \colon \Lambda^i\mathfrak{g}_p &\to\Omega^{n-k-i}(M),\\ q&\mapsto -\varsigma(k)f_{i+k}(q\wedge p).
	\end{align*}
\end{prop}
\begin{remark} In the context of multisymplectic geometry and ``weak comoments'' (cf. \cite{Herman2018}) and ``multimoments'' (cf. \cite{Madsen2013}),  the subspace $ Z_{k}(\mathfrak{g})$ is often referred to as ``the Lie kernel''.
\end{remark}
\begin{prop}[Remark 3.9 in \cite{Ryvkin2016}]\label{prop:indmomap_equiv} 
	If the comoment $(f) \colon \mathfrak{g} \to L_{\infty}(M,\omega)$ is also $G$-equivariant, then the map $(f^p)$ defined in proposition \ref{prop:indmomap} is $G_p^0$-equivariant.\\
	Another equivariant comoment
	for the action of $G_p^0$ on $(M,\iota(v_p)\omega)$ is given by $(i=1,\dots,n-k)$:
	\begin{align*}
	f^p_i: \Lambda^i\mathfrak{g}_p &\to\Omega^{n-k-i}(M),\\ q&\mapsto (-1)^{k}\iota(v_q)(f_{k}(p)).
	\end{align*}
	which, in general, may differ from the one given in Proposition \ref{prop:indmomap}. 
\end{prop}
\noindent

\vspace{1em}
Provided certain conditions are met, it is possible to induce a comoment on an invariant submanifold of $M$ even if the obvious pullback vanishes (e.g.~when  $\omega$ is a top dimensional form).
In what follows we will make use of the following corollary subsuming the contents of the previous propositions:
\begin{cor}\label{cor:inducedmachinery}
	Let $G\action(M,\omega)$ be a multisymplectic group action.
	If there exists:
	\begin{itemize}[topsep=1pt,itemsep=0pt, partopsep=1pt]
		\item another multisymplectic manifold $(N,\eta)$ containing $M$ as a $G$-invariant embedding $j:M\hookrightarrow N$;
		\item a Lie group $H \supset G$ containing $G$ as a Lie subgroup;
		\item a multisymplectic action $H \action (N,\eta)$ with equivariant comoment $s:\mathfrak{h}\to L_\infty(N,\eta)$;
		\item an element $p\in Z_k(\mathfrak{h})$ in the Lie kernel of $\mathfrak{h}$ such that $G \subset H_p$ and $\omega = j^\ast \iota_p \eta$;
	\end{itemize}
	then the action $G\action(M,\omega)$ 	admits an equivariant comoment,	given by $(i=1,\dots,n-k)$:
	\begin{align*}
		f_i:\Lambda^i\mathfrak{g}&\to\Omega^{n-k-i}(M),\\ 
		q&\mapsto (-1)^{k}j^\ast\iota(v_q)(s_{k}(p)).
	\end{align*} 	
\end{cor}
\begin{proof}
	Starting from the given comoment $(s)$ it is possible to construct another comoment $(s^p)$ resorting to proposition \ref{prop:indmomap_equiv}.
	The sought comoment descends from $(s^p)$ via the consecutive application of propositions \ref{prop:restrict} and \ref{prop:NenMGinvariant}
	\begin{center}
	\begin{tikzcd}[column sep=large]
		H \ar[symbol=\action]{r} & (N,\eta)
		& & \mathfrak{h} \ar["(s)"]{r} & L_{\infty}(N,\eta)
		\\
		H_p \ar[symbol=\action]{r} \ar[hook]{u}& (N,\iota_{v_p}\eta)  
		& & \mathfrak{h}_p\ar[hook]{u} \ar[dashed,"(s^p)"]{r}[swap]{Prop. \ref{prop:indmomap_equiv}} & L_{\infty}(N,\iota_{v_p}\eta) \ar[equal]{d} 
		\\
		G \ar[symbol=\action]{r} \ar[hook]{u}& (N,\iota_{v_p}\eta)   \ar[equal]{u}
		& & \mathfrak{h}\ar[hook]{u} \ar[dashed,"Prop. \ref{prop:restrict}"']{r} & L_{\infty}(N,\iota_{v_p}\eta) \ar["j^\ast"]{d} 
		\\
		G \ar[symbol=\action]{r} \ar[equal]{u}& (M,\omega =j^\ast \iota_{v_p}\eta) \ar[hook]{u}  
		& & \mathfrak{g}\ar[equal]{u} \ar[dashed,"Prop. \ref{prop:NenMGinvariant}"']{r}& L_{\infty}(M,\omega)  
	\end{tikzcd}
	\end{center}
	together with the observation that if the starting comoment $(s)$ is equivariant then the induced maps are such.
\end{proof}

\subsection{The action of $SO(n)$ on $S^n$}\label{subsecson}
The goal of this section is to give an explicit construction for the comoment of the action $SO(n) \action S^n$ by resorting to Corollary \ref{cor:inducedmachinery}.

\vspace{1ex}
In what follows, we will consider the standard $SO(n+1)$-invariant embedding $ j: S^n \to \mathbb{R}^{n+1}$ of $S^n$ as the sphere with unit radius
and consider the linear action of the group $SO(n)$ on $\mathbb{R}^{n+1}$ as the subgroup of special orthogonal linear transformations fixing the axis $x^0$.
\\
In~$\mathbb{R}^{n+1}$ we consider the standard coordinates~$x = (x^0,\dots,x^n)$ and the corresponding volume form~$\Vol = d x^0\wedge\dots \wedge dx^{n}$. Furthermore, we will make use of the following notation  
\begin{displaymath}
r = \sqrt{(x^1)^2 + \dots + (x^n)^2} \qquad R = \sqrt{(x^0)^2 + r^2} .
\end{displaymath}

\vspace{1ex}
Recall that the volume form on the unit sphere embedded in $\mathbb{R}^{n+1}$ is given by 
$$ \omega = j^\ast \iota_E ~\Vol$$
where $E$ is the Euler vector field.
$E$ can be seen as the fundamental vector field of the action
\begin{align*}
	\vartheta: \mathbb{R} \times \mathbb{R}^{n+1}
	&\to \mathbb{R}^{n+1}\\
	(\lambda,x)&\mapsto e^{\lambda} x
\end{align*} 
of $\mathbb{R}$ on the Euclidean space via dilations,
that is the linear action generated by the identity matrix $\id_{n+1} \in \mathfrak{gl}(n,\mathbb{R}^{n+1})$, i.e. $ E= v_{\id_{n+1}} = \sum_i x^i\, \partial_i$.

Let us call $H = SO(n)\times\mathbb{R}$ the subgroup of $GL(n,\mathbb{R}^{n+1})$ generated by
\begin{displaymath}
	\mathfrak{h} = 
	\left\lbrace
		\begin{bmatrix} 
			1 & 0 \\ 
			0 & a 
		\end{bmatrix}
		\; \vert \: a \in \mathfrak{so}(n)
	\right\rbrace 
	\oplus \langle \id_{n+1} \rangle \simeq \mathfrak{so}(n)\oplus \mathbb{R}
	~.
\end{displaymath}
The group $H$ acts linearly on $\mathbb{R}^{n+1}$ through the standard infinitesimal action
\begin{center}
	\begin{tikzcd}[column sep= small,row sep=0ex]
		v \colon& \mathfrak{h} 	\arrow[r]
			& \mathfrak{X}(\mathbb{R}^{n+1}) \\
			& A \arrow[r, mapsto]& 	\displaystyle\sum_{i,j} [A]_{i j} ~ x^j \partial_i
	\end{tikzcd}
\end{center}
where $[A]_{i j}$ denotes the $i j$ entry of the matrix $A$.

\begin{lemma}\label{lem:rescaledvolume}
The differential form
\begin{displaymath}
 \eta = \rescaling \,\Vol \in \Omega^{n+1}(\mathbb{R}^{n+1} \setminus \{0\})
 \qquad \text{with} \qquad \rescaling = \dfrac{1}{R^{n+1}}
\end{displaymath}	
is multisymplectic on $N = (\mathbb{R}^{n+1} \setminus \{0\})$, invariant under the action $H \action N$ and restricts to the standard volume form on the unit sphere.
\end{lemma}
\begin{proof}
	Multisymplecticity follows from the closure and non degeneracy of $\Vol$ together with the property that $\rescaling$ never vanishes.
	\\ 
	The form is clearly $SO(n)$-invariant because $\rescaling$ depends only on $R$. The $H$-invariance follows from the invariance along the Euler vector field:
$$
	\mathcal{L}_E (\rescaling~\Vol) = (\mathcal{L}_E \rescaling + (n+1) \rescaling ) \Vol
$$
and
$$
	\mathcal{L}_E \rescaling = \sum_i x^i \dfrac{\partial}{\partial x^i} R^{-(n+1)}	
	= -\dfrac{n+1}{2}\sum_i 2 (x^i)^2 R^{-(n+1)-1} =-(n+1) \rescaling	
$$
Finally, $\eta$ restricts to the volume form on $S^n$ because $j^\ast \rescaling = 1$.
\end{proof}

The function $\rescaling \in C^\infty(\mathbb{R}^{n+1}\setminus\{0\})$ is precisely the scaling factor that makes the Euclidean volume invariant with respect to the extended group $SO(n+1)\times\mathbb{R}$.
The problem to find explicitly a comoment:
\begin{displaymath}
s \colon \mathfrak{h} \rightarrow L^{\infty} \left(\mathbb{R}^{n+1}\setminus\{0\},\eta = \rescaling \, \Vol\right)
\end{displaymath}
can be solved by exhibiting an $H$-invariant primitive of  the rescaled volume $\eta$ and then resorting to Lemma \ref{lem:extexact}.

\begin{lemma}\label{lem:uglyprimitive}
	The differential (n+1)-form $\eta$ admits an $H$-invariant potential $n$-form: 
	\begin{displaymath}
		\beta = (\hat{\varphi}~x^0)~dx^1\wedge\dots\wedge dx^n
	\end{displaymath}	
where $\hat{\varphi}\in C^\infty (\mathbb{R}^{n+1}\setminus{0})$ depends only on the cylindrical coordinates $(x_0,r)$ and it is given by
	\begin{equation}\label{eq:uglyprimitive}
	\hat{\varphi}(x^0,r)  = 	
	\begin{cases}
		\frac{1}{\left((x^0)^2 + r^2\right)^{\frac{n+1}{2}}}
		\left(x^0 (n+1) - r \arctan\left(\dfrac{x^0}{r}\right)\right)
		&\quad r \neq 0
		\\
		(n+1)\frac{1}{\vert x^0 \vert^n} 
		&\quad r=0,\, x^0 > 0
		\\  
		-(n+1)\frac{1}{\vert x^0 \vert^n} 
		&\quad r=0,\, x^0 < 0
	\end{cases}
	\end{equation}
\end{lemma}
\begin{proof}
	Let us start from the following ansatz 
		$$\beta =\iota_{(x^0\partial_0)} \varphi(x^0,r) \eta $$
	for a potential $n$-form of the scaled volume $\eta$ as defined \ref{lem:rescaledvolume}.
	At this point, $\varphi$ is an arbitrary smooth function depending only on the cylindrical parameters $(x^0,r)$.
	\\
	Being $x^0 \partial_0$ the fundamental vector field of 
	\begin{displaymath}
		\zeta =
			\begin{bmatrix} 
				1 & 0 \\ 
				0 & 0_n 
			\end{bmatrix}
		\in \mathfrak{gl}(n+1,\mathbb{R})
		,
\end{displaymath}
one gets
\begin{displaymath}
\mathcal{L}_{v_\xi} \beta = \left(
\iota(\cancel{v_{[\xi,\zeta]}}) + \iota({v_\zeta}) \mathcal{L}_{v_\xi}
\right)
\varphi \, \eta = 0
\qquad \forall \xi \in \mathfrak{so}(n),
\end{displaymath}
because the $(n+1)$-form $\varphi \, \eta $ depends only on $(x^0,r)$,
i.e. $\beta$ is $SO(n)$ invariant. On the other hand, one has:
\begin{displaymath}
	\begin{split}
		\mathcal{L}_E \beta &= \left(
		\cancel{\iota_{[\id,\zeta]}} + \iota_{x_0 \partial_0} \mathcal{L}_{E}
		\right)
		\varphi \, \rescaling \, \Vol = 
		\\
		&= \iota_{x_0 \partial_0} \left[
		\left(\mathcal{L}_E \varphi \right)\, \rescaling \, \Vol 
		+ \varphi \cancel{\mathcal{L}_E\,\rescaling\Vol}
		\right]
	\end{split}
\end{displaymath}
and
\begin{displaymath}
	\begin{split}
		\textrm{d} \beta &=
		\left(
		\dfrac{\partial \varphi}{\partial x^0}\,\rescaling\,x^0 +
		\varphi\, x^0 \, \dfrac{\partial \rescaling}{\partial x^0} + 
		\varphi \rescaling \right) \Vol 
		\\
		&= \left( \partial_0 \varphi \, x^0 - (n+1)\dfrac{(x^0)^2}{(r^2 + (x^0)^2)} 
		+ \varphi \right)\, \rescaling\,\Vol 
		~.
	\end{split}
\end{displaymath}
Therefore, in order for $\beta$ to be $G-$invariant primitive of $\omega$, the following conditions on $\varphi$ have to be met:
\begin{displaymath}
	\begin{cases}
		\mathcal{L}_E \varphi = r\,\partial_r\,\varphi + x^0\,\partial_0\,\varphi = 0 
		\\
		x^0\,\partial_0 \varphi   - (n+1)\dfrac{(x^0)^2}{(r^2 + (x^0)^2)} + \varphi = 1
	\end{cases}
\end{displaymath} 
The general solution of this system reads:
\begin{displaymath}
	\varphi(x^0,r) = - \dfrac{r}{x^0}\arctan\left(\dfrac{x^0}{r}\right)(n+1) + n + 2
\end{displaymath}
which is a smooth function defined on 
$\left\{ x \in \mathbb{R}^{n+1} ~\vert x^0 \neq 0, r \neq 0\right\}$.
Recalling that 
$$\lim_{y\to 0}\dfrac{\arctan(y)}{y}=1 \qquad \lim_{y\to \infty}\dfrac{\arctan(y)}{y}=0 ~,$$
one can see that the limits to all the critical points of $\varphi$, except $(x^0,r)=(0,0)$, are finite.
Hence, considering the unique smooth extension $\hat{\varphi}\in C^\infty(\mathbb{R}^{n+1}\setminus\{0\})$ of $\varphi$, given explicitly by equation (\ref{eq:uglyprimitive}), we conclude that 
	\begin{displaymath}
		\beta = (x^0~\hat{\varphi})~dx^1\wedge\dots\wedge dx^n
	\end{displaymath}	
is the sought $H$-invariant primitive.
\end{proof}

\begin{prop}\label{Prop:SonSn}
	A comoment for the action $SO(n) \action \left( S^{n}, \omega\right)$, for $n \geq 2$, is given by
	\begin{align*}
	f_i:\Lambda^i\mathfrak{so}(n)&\to\Omega^{n-1-i}(S^n),\\ 
	q&\mapsto -j^\ast\iota(v_q)(\iota_E \beta).
	\end{align*} 
	where $\beta$ is the primitive given in Lemma \ref{lem:uglyprimitive}.
\end{prop}
\begin{proof}
	The statement follows directly from Corollary \ref{cor:inducedmachinery} upon considering

	\begin{align*}
		(N,\eta) = (\mathbb{R}^{n+1},\rescaling\Vol ) \qquad\quad 
		&(M,\omega) = (S^n, j^\ast \iota_E \Vol) 
		\\
		p= \id_{(n+1)} \in Z_1(\mathfrak{h}) \qquad\quad
		& H = SO(n)\times \mathbb{R} = H_E \supset SO(n)		
	\end{align*}
	and noting that an explicit comoment for the $H$-action is given by Lemma \ref{lem:extexact} via employment of the primitive constructed in Lemma \ref{lem:uglyprimitive}.
\end{proof}

\begin{remark}
	Proposition \ref{Prop:SonSn} extends to spheres of arbitrary dimension a similar result given in \cite[Paragraph 8.3.2]{Callies2016} up to dimension $5$.
\end{remark}

%----------------------------------------------------------------------------
\section{Transitive multisymplectic group actions on spheres}\label{sectra}
%----------------------------------------------------------------------------

The goal of this section is to prove the following theorem:

\begin{thm}\label{thm:surprise}
Let $G$ be a compact Lie group acting multisymplectically, transitively and effectively on $S^n$ equipped with the standard volume form, then the action admits a comoment if and only if $n$ is even. 
\end{thm}

\begin{proof}[Proof of Theorem \ref{thm:surprise}] 
The effective transitive compact connected group actions on spheres have been classified \cite{MR0008817, MR0034768,MR0029915}, for an overview of the results cf. \cite{MR2371700}. Essentially, one has the following list, where $G/H=S^n$ means that $G$ acts transitively on $S^n$ with isotropy $H$.
\begin{itemize}
\item $SO(n)/SO(n-1)=S^{n-1}$
\item $SU(n)/SU(n-1)=U(n)/U(n-1)=S^{2n-1}$
\item $Sp(n)Sp(1)/Sp(n-1)Sp(1)=Sp(n)U(1)/Sp(n-1)U(1)=Sp(n)/Sp(n-1)=S^{4n-1}$
\item $G_2/SU(3)=S^6$
\item $Spin(7)/G_2=S^7$
\item $Spin(9)/Spin(7)=S^{15}$.
\end{itemize}
From Propostion \ref{prop:restrict} we know that if $G$ admits a comoment and $\tilde G\subset G$ is a subgroup, then $\tilde G$ admits a comoment. 
Thus it will suffice to prove the following statements:
\begin{enumerate}
\item The action of $SU(n)$ on $S^{2n-1}$ does not admit a comoment. As $SU(n)\subset U(n)\subset SO(2n)$, from this we will automatically get the statements, that $U(n)$ and $SO(2n)$ do not admit a comoment when acting on $S^{2n-1}$. Moreover, as $SU(4)\subset Spin(7)$, this implies that $Spin(7)$ does not admit a comoment, when acting on $S^7$.
\item The action of $Sp(n)$ on $S^{4n-1}$ does not admit a comoment. Hence, neither $Sp(n)U(1)$ nor $Sp(n)Sp(1)$ do. 
\item Spin(9) does not admit a comoment, when acting on $S^{15}$.

\item $SO(2n+1)$ has a comoment when acting on $S^{2n}$. As $G_2\subset SO(7)$, this implies that $G_2$ admits a comoment when acting on $S^6$.

\end{enumerate}
Hence, we can prove the theorem by using Lemma \ref{lem:core}, after proving the following  Proposition \ref{prop:stiefel} and Proposition \ref{prop:annoying}. 
\end{proof}

\begin{prop}\label{prop:stiefel}
Let $\omega_n$ be the volume of $S^n$ and $N$ the north pole.
\begin{itemize}
\item Let $\vartheta_N:SU(n)\to S^{2n-1}$. Then $\vartheta_N^*[\omega_{2n-1}]\neq 0$.
\item Let $\vartheta_N:Sp(n)\to S^{4n-1}$. Then $\vartheta_N^*[\omega_{4n-1}]\neq 0$.
\item Let $\vartheta_N:Spin(9)\to S^{15}$. Then $\vartheta_N^*[\omega_{15}]\neq 0$.
\end{itemize}
\end{prop}
\begin{proof}
For the first two cases we refer to the calculations in the proof of \cite[Corollary 4D.3]{MR1867354}, where it is shown that the generator of the sphere is turned into a generator of $SU(n)$ (resp. $Sp(n)$) via pullback (in our case via $\vartheta_N$). For the third case we can proceed analogously: The fiber bundle $Spin(7)\to Spin(9)\to S^{15}$ has a 14-connected base, so the pair $(Spin(9),Spin(7))$ is 14-connected and the maps $H^i(Spin(9))\to H^i(Spin(7))$ are isomorphisms for $i\leq 13$. As $H^\bullet(Spin(7))$ is generated by (wedges of) elements in these degrees, this implies that the Leray-Hirsch-theorem \cite[Theorem 4D.1]{MR1867354} can be applied. This means that $\vartheta_N^*[\omega_{15}]$ is a generator and thus nonzero in $H^{15}(Spin(9))$.
\end{proof}

In the case of $SO(2n+1)$ the Leray-Hirsch theorem can not be applied, as $SO(2n)$ has a class in degree $2n-1$ which does not come from any class in $SO(2n+1)$, (cf. e.g. \cite[Theorem 3D.4]{MR1867354} or Appendix \ref{appendix:details_son}). In fact we have the following:
\begin{prop}\label{prop:annoying}
Let $\omega_{2n}$ be the volume of $S^{2n}$ and $N$ the north pole. Let $\vartheta_N:SO(2n+1)\to S^{2n}$. Then $\vartheta_N^*[\omega_{2n}]= 0$.
\end{prop}
\begin{proof}
Let $i:SO(2n)\to SO(2n+1)$ be the inclusion. The cohomologies of $SO(2n+1)$ and $SO(2n)$ are isomorphic up to degree $2n$ and $i^*:H^{2n}(SO(2n+1))\to H^{2n}(SO(2n))$ is an isomorphism. The class $[i^*\vartheta_N^*\omega_{2n}]$ is the obstruction against a comoment for the $SO(2n)$-action on $S^{2n}$. We know from Proposition \ref{lem:core}, that this action admits a comoment, i.e. $[i^*\vartheta_N^*\omega_{2n}]=0\in H^{2n}(SO(2n))$. But as $i^*$ is an isomorphism, this implies that $[\vartheta_N^*\omega_{2n}]=0\in H^{2n}(SO(2n+1))$.
\end{proof}

\subsection{Explicit comoments for $SO(n+1)$ on $S^{n}$}\label{subsectra}
Giving an explicit expression for a comoment of $SO(n+1) \action S^{n}$ requires to find iteratively, for $k \in (1,\dots,n-1)$ and  for all $p \in \Lambda^k \mathfrak{so}(n)$, a primitive, denoted as $f_k(p)$, of the closed differential $(n-k)$-form
\begin{equation}\label{eq:comomentasprimitive}
	\mu_k (p) = 	-f_{k-1} (\partial p) - \varsigma(k) \iota(v_p) \omega 
\end{equation}
with $f_0 = 0$ and satisfying the following constraint
\begin{equation} \label{eq:comomentconstraints}
	f_n(\partial p ) = - \varsigma(k+1) \iota_{v_p} \omega ~.
\end{equation}

As $H^0(S^n)$ and $H^{n}(S^n)$ are the only non trivial cohomology groups, it is always possible to find primitives of $\mu_k (p)$.
The only thing that could fail, and actually fails when $n$ is odd, is the fulfilment of condition \eqref{eq:comomentconstraints}.
In the latter case, it is however possible to consider an extension of $\mathfrak g$ to a suitable Lie-$n$ algebra and consider \emph{Lie-$n$ homotopy comoment map} instead of our notion of comoments (See \cite{Callies2016} or \cite{Mammadova2019} for the explicit case of $n=4$).

\vspace{1em}
Instead of dealing with the analytical problem of finding explicit potentials for the form $\mu_k (p)$, let us translate the problem in a more algebraic fashion focusing on the particular structure of the Chevalley-Eilenberg complex of $\mathfrak{so}(n)$.
\\
In general, it is fairly easy to express the action of a comoment on boundaries:
\begin{lemma}[Comoments on boundaries]
	Let $v:\mathfrak{g}\to (M,\omega)$ be a multisymplectic infinitesimal action.\\
	Let $F^k: B_k(\mathfrak{g}) \to \Lambda^{k+1}\mathfrak{g}$ such that $\partial \circ F^k = \text{id}_{B^k}$
	,i.e. $F^k$ gives a choice of representative of a primitive for every $k$-boundary of $\mathfrak{g}$.\\
	Then, the function $ 	f_k \colon B_k(\mathfrak{g}) \to\Omega^{n-k}(M)$ defined as
	\begin{displaymath}
		f_k(p) = -\varsigma(k+1) \iota(v_{F^k(p)}) \omega
	\end{displaymath}
	satisfies equation (\ref{eq:fk_hcmm}) defining the $k$-th component for a comoment of the infinitesimal action, for every boundary $p$.
\end{lemma}
\begin{proof}
	It is a straightforward application of Lemma \ref{lemma:multicartan} together with the multisymplecticity of the infinitesimal action:
	\begin{displaymath}
		\begin{split}
		d f_k (p) & = 
		-\varsigma(k+1) d \iota(v_{F^k(p)}) \omega =
		-(-1)^{k+1}\varsigma(k+1) \iota(v_{\partial F^k(p)}) \omega =\\
		& = \cancel{- f_{k-1}(\partial p)} - \varsigma(k) \iota_{v_p} \omega.
		\end{split}
	\end{displaymath}
\end{proof}
\begin{remark}
Note that the costraint given by equation (\ref{eq:comomentconstraints}) is precisely the requirement that action on boundaries of the highest component $f_n$ of the comoment $(f)$ is independent from the choice of $F^n$.
\end{remark}
In other words, finding the action of the component $f_k$ of comoment $(f)$ on boundaries is tantamount to finding a function $F^k: B_k(\mathfrak{g}) \to \Lambda^{k+1}\mathfrak{g}$ mapping a boundary $p$ to a specific primitive $q$.
\\
Note that this is nothing more than replacing the problem of finding a potential of an exact differential form to the one of finding a primitive of a boundary in the  CE-complex.

It follows from the previous lemma that the $k$-th component of the comoment is completely determined by its action on boundaries when $H_k(\mathfrak{g})=0$:
\begin{cor}
	Consider a Lie algebra with $H_k(\mathfrak{g})=0$ and fix a choice of representatives $F^k$ and $F^{k-1}$ as before.
	Then, the function $ 	f^k \colon \Lambda^k \mathfrak{g} \to\Omega^{n-k}(M)$ defined as
	\begin{displaymath}
	f_k(q) = \varsigma(k+1)\iota(v_{F^k(F^{k-1}(\partial q) -q)})\omega
	\end{displaymath}
	satisfies equation (\ref{eq:fk_hcmm}) for every chain $q \in \Lambda^k\mathfrak{g}$.
\end{cor}
\begin{proof}
	Consider a function $f_{k-1}$ defined through $F^{k-1}$ according to the previous lemma.
	For every cycle $q \in \Lambda^k\mathfrak{g}$ we get
	\begin{displaymath}
		- f_{k-1}(\partial q) -\varsigma(k)\iota_{v_q}\omega =
		\varsigma(k)[\iota(v_{F^{k-1}(\partial q)}) - \iota(v_{q})] \omega=
		\varsigma(k)\iota( v_r )\omega		
	\end{displaymath}
	where $r =(F^{k-1}(\partial q) - q)$ is closed, hence exact.
	Again from Lemma \ref{lemma:multicartan} follows that the right hand side it is equal to
	\begin{displaymath}
		\varsigma(k)\iota(v_{\partial F^k(r)} \omega)	=
		\varsigma(k)(-1)^{k+1} d \iota(v_{F^k(r)}) \omega =
		d f_k(q).
	\end{displaymath}
\end{proof}

An explicit construction of a comoment is generally more delicate in presence of cycles that are not boundaries.
Nevertheless, we know from theorem \ref{thm:son-cohomology} that the first two homology group of $\mathfrak{so}(n)$ are trivial, therefore it is easy to give the first two components of the comoment.
From the linearity of $f_k$, it is clear that we only need to give its action on the standard basis of the finite-dimensional vector space $\mathfrak{so}(n)$ defined in appendix, equation \eqref{eq:standard-basis}.

\paragraph{$f_1$ for any $SO(n)$.}

Since all 1-chains in the CE complex are automatically cycles, $H^1(\mathfrak{so}(n))=0$ implies that all elements of $\mathfrak{so}(n)$ are boundaries.
\\
Formula \eqref{eq:reductionformula} suggests a natural choice for the function $F^1$ when acting on elements of the standard basis:
\begin{displaymath}
	F^1(A_{a b}) = -\sum_{k=1}^n \dfrac{1}{n-2}  A_{k a}\wedge A_{k b}
\end{displaymath}
Therefore the first component of the comoment is given by
\begin{displaymath}
	f_1 (A_{a b}) = - \iota(v_{F^1(A_{a b})}) \omega =
	\dfrac{1}{n-2}\sum_{k=0}^n \iota(v_{A_{k b}})\iota(v_{A_{k a}})\omega.
\end{displaymath}
\begin{example}
In the three-dimensional case, denoting the three generators of $\mathfrak{so}(3)$ as $l_x,l_y,l_z$ (see appendix), one gets
\begin{displaymath}
	F^1 (l_i) = - \dfrac{1}{2}\sum_{j,k=1}^3\epsilon_{i j k} l_j \wedge l_k,
\end{displaymath}
where $\epsilon_{i j k}$ is the Levi-Civita symbol, and
\begin{displaymath}
	f_1(l_z) = \iota(v_{l_y \wedge l_z}) \omega = j^\ast \iota(E \wedge v_{l_y}\wedge v_{l_z}) dx^{123}.
\end{displaymath}
\end{example}

\paragraph{$f_2$ for any $SO(n)$.}
In this case there are two subsets of generators of $\Lambda^2 \mathfrak{so}(n)$ to consider:
\begin{displaymath}
	\begin{cases}
    p = A_{a b} \wedge A_{c d} \xmapsto{\,\partial\,} 0 & \text{for } a,b,c,d \text{ different} \\
    q = A_{j a} \wedge A_{j b} \xmapsto{\,\partial\,} -A_{a b} & \text{for } j,a,b \text{ different} 
  \end{cases}
\end{displaymath}
The first ones are boundaries and a primitive can be given as follows
\begin{displaymath}
F^2 (A_{a b} \wedge A_{c d}) = \dfrac{n-2}{4}\left( F^1(A_{a b}) \wedge A_{c d} - A_{a b}\wedge F^1(A_{c d}) \right)
\end{displaymath}
In the second case, we need to find a primitive of
\begin{displaymath}
	\begin{split}
		 F^1(\partial q) -q &= -F^1(A_{a b}) - A_{j a} \wedge A_{j b}
		 = \dfrac{1}{n-2}\sum_{k=0}^n( A_{k a}\wedge A_{k b} - A_{j a} \wedge A_{j b})
		 = \\
		 &=\dfrac{1}{n-2}\sum_{k=0}^n \partial (A_{k a}\wedge A_{j b}\wedge A_{k j}) =
		  \partial \left(\dfrac{1}{n-2}\sum_{k=0}^n (A_{k a}\wedge A_{j b}\wedge A_{k j}) \right)
	\end{split}
\end{displaymath}
The last equality suggests the following choice
\begin{displaymath}
 F^2(F^1(\partial q) -q) = \left(\dfrac{1}{n-2}\sum_{k=0}^n (A_{k a}\wedge A_{j b}\wedge A_{k j}) \right).
\end{displaymath}
Finally, one gets
\begin{displaymath}
\begin{split}
	f_2(A_{a b} \wedge A_{c d}) &= \dfrac{n-2}{4}\left(
	\iota(v_{F^1(A_{a b}) \wedge A_{c d}}) - \iota(v_{A_{a b}\wedge F^1(A_{c d})})\right)\omega\\
	f_2(A_{j a} \wedge A_{j b}) &= \dfrac{-1}{n-2}
	\sum_{k=0}^n	
	\left(\iota(v_{A_{k a}\wedge A_{j b}\wedge A_{k j}})\right)\omega.
\end{split}
\end{displaymath}

\subsection{Explicit comoment for $G_2$ on $(S^6,\phi)$}
We finish by providing a nice example of comoments for non-volume forms on spheres.\\

Recall that	$G_2$ is a subgroup of $SO(7)$ acting transitively and multisymplectically on $S^6$ with the standard volume. Therefore, according to Theorem \ref{thm:surprise}, the action $G_2\action (S^6,\omega)$ admits a comoment.
	
This group can be explicitly defined as the subgroup of $GL(7,\mathbb{R})$ preserving the multisymplectic $3$-form
\begin{equation}
	\phi
	=d x^{123}+ d x^{145}+ d x^{167}+ d x^{246}- d x^{257}- d x^{356}-d x^{347},
\end{equation}
where~$x = (x^i)$ are the standard coordinates on~$\mathbb{R}^7$ 
and~$d x^{ijk} = d x^i \wedge d x^j \wedge d x^k$. 
(See \cite{MR1939543} for further remarks on $G_2$-homogeneous multisymplectic forms and \cite{MR2253159} for details on the $G_2$-manifold $S^6$).

Considering the multisymplectic structure $j^\ast\phi$ on $S^6$, where $j$ is the inclusion of the sphere in $\mathbb{R}^7$, instead of the standard volume, it is possible to give an explicit comoment for the action of $G_2$:
\begin{lemma}
	The action $G_2 \circlearrowright (S^6, j^\ast\phi)$ admits an equivariant comoment given by $(k=1,2)$:
\begin{align*}
	f_{k} \colon \Lambda^k\mathfrak{g}_2 &\to \Omega^{2-k}(M),\\
	q&\mapsto (-1)^{k-1}j^\ast \iota_{v_q}\iota_E \dfrac{\phi}{3}	
\end{align*}	
\end{lemma}	
\begin{proof}
	It follows from Lemma \ref{lem:extexact}, noting that $(\frac{1}{3} \iota_E \phi)$ is a $G_2$ invariant primitive of $\phi$ in $\mathbb{R}^3$.
\end{proof}

%----------------------------------------------------------------------------------------------------------------------------------
\appendix
\section{Appendix}
%----------------------------------------------------------------------------------------------------------------------------------
\subsection{Useful formulas in Cartan calculus}
We make use of the following formula:
\begin{lemma}[Multi-Cartan magic formula]\label{lemma:multicartan}
	\begin{equation}
	\begin{split}
		(-1)^m \textrm{d} \iota(x_1\wedge\dots\wedge x_m) &= \iota(x_1\wedge\dots\wedge x_m) \textrm{d} +\\
		&+ \iota(\partial \, x_1\wedge\dots\wedge x_m) +\\
		&+ \sum_{k=1}^{m} (-1)^k  \iota( x_1\wedge\dots\wedge  \hat{x_k}\, \wedge\dots\wedge x_m) \mathcal{L}_{x_k}.
	\end{split}
	\end{equation}
\end{lemma}
\begin{proof}
	See lemma 3.4 in \cite{Madsen2013}.
\end{proof}
 
\begin{defi}\label{multidef}
Given a differential form $\Omega\in \Omega^\bullet(M)$ and a multi-vector field $Y\in \Gamma(\Lambda^m TM)$, the \emph{Lie derivative of $\Omega$ along $Y$} is defined as the graded commutator
\begin{equation}
	\mathcal{L}_Y\Omega:=d\iota_Y\Omega-(-1)^m\iota_Y d\Omega.
\end{equation}
\end{defi}
\begin{remark}\label{rem:Lieder}
	This definition allows to combine the first and last term in the above formula into a Lie derivative. Hence the formula of Lemma \ref{lemma:multicartan} can be also written as
	\begin{displaymath}
		\mathcal{L}_{v_{1} \wedge \dots
		\wedge {v}_{m}}\Omega 
		=(-1)^m \biggr[
		\iota(\partial(v_{1}\wedge\dots\wedge v_{m}))\Omega+\sum_{1=1}^{m} (-1)^{i} 
		\iota( v_{1} \wedge \dots
		\wedge \hat{v}_{i} \wedge \dots \wedge {v}_{m})\mathcal{L}_{v_i}\Omega\biggr].
	\end{displaymath}	 
\end{remark} 

\begin{lemma}
	\begin{equation}\label{eq:multiliecartan}
		\mathcal{L}_v \iota (x_1 \wedge \dots \wedge x_k) =
		\iota([v,x_1\wedge\dots\wedge x_k ] ) +
		\iota(x_1 \wedge \dots \wedge x_k) \mathcal{L}_v.
	\end{equation}
\end{lemma}
\begin{proof}
	It is simply an iterated application of the Cartan's commutation relation:
	\begin{displaymath}
		\mathcal{L}_x \iota_y = \iota_y \mathcal{L}_x + \iota_{[x,y]}.
	\end{displaymath}
	together with expression \eqref{eq:adjointactionwedge}.
\end{proof}

\subsection{Some technical details about $\mathfrak{so}(n)$}\label{appendix:details_son}

Recall that $\mathfrak{so}(n)$ is the Lie sub-algebra of $\mathfrak{gl}(n,\mathbb{R})$ consisting of all skew-symmetric square matrices. A basis can be constructed as follows:
\begin{equation}\label{eq:standard-basis}
	\mathcal{B}\coloneqq \big\lbrace 	A_{a b} = (-1)^{1+a+b} \left( E_{a b} - E_{b a}\right)
	\quad \vert \quad 1<a<b\leq n \big\rbrace
\end{equation}
where $E_{a b}$ is the matrix with all entries equal to zero and entry $(a,b)$ equal to one.
\\
The fundamental vector field of $A_{a b}$ associated to the linear action of $SO(n)$ on $\mathbb{R}^n$ reads as follows:
\begin{displaymath}
	v_{A_{a b}}= \sum_{i,j}[A_{a b}]_{i j}x^j \partial_i  = (-1)^{1+a+b}\left(x^a \partial_b - x^b \partial_a\right)
\end{displaymath}
\begin{example}\label{ex:angularmomenta}
In $\mathbb{R}^3$ we have three matrices
\begin{displaymath}
	l_x = A_{1\, 2} = \begin{bmatrix} 0 & 1 & 0 \\ -1 & 0 & 0 \\ 0 & 0 & 0 \end{bmatrix} \qquad
	l_y = A_{1\, 3} = \begin{bmatrix} 0 & 0 & -1 \\ 0 & 0 & 0 \\ 1 & 0 & 0 \end{bmatrix} \qquad
	l_z = A_{2\, 3} = \begin{bmatrix} 0 & 0 & 0 \\ 0 & 0 & 1 \\ 0 & -1 & 0 \end{bmatrix}
~.
\end{displaymath}
\end{example}
\noindent Using such a basis, the structure constants read as follows:
\begin{lemma}
	\begin{displaymath}
	\begin{split}
		[A_{a b}, A_{c d}] =& (-1)^{(b+c+1)}\delta_{b c} A_{a d} +
		(-1)^{(a+d+1)}\delta_{a d} A_{b c} +\\
		&	(-1)^{(d+b+1)}\delta_{d b} A_{a c} +
		(-1)^{(a+c+1)}\delta_{c a} A_{d b}	
	\end{split}
	\end{displaymath}
		in particular:
	\begin{equation}\label{eq:reductionformula}
		[A_{k a}, A_{k b}] = A_{a b}
	\end{equation}
	for all $k \neq a,b$.
\end{lemma}
\begin{thm}\label{thm:son-cohomology}%[Cohomology groups of $SO(n)$]
The cohomology groups of $SO(n)$ can be described as follows:
\begin{displaymath}
	H^\bullet(SO(n); \mathbb{R}) \simeq
	\begin{cases}
		\big\langle \lbrace a_{4\,i-1} \vert 1\leq i\leq k \rbrace \cup a^\prime_{2\,k+1} \big\rangle & n=2k+2\\
		\big\langle \lbrace a_{4\,i-1} \vert 1\leq i\leq k \rbrace \big\rangle & n=2k+1
	\end{cases}
\end{displaymath}
where $\langle \dots \rangle$ denotes the exterior algebra generated by the elements in the set and subscripts denote degrees, so $a_i \in H^i$ and $a'_{2k+1} \in H^{2k+1}$.
\end{thm}
\begin{proof}
	See section 3D in \cite{MR1867354}.
\end{proof}

%------------------------------------------------------------------------------------------------
% Bibliography (BibTex)
% https://arxiv.org/hypertex/bibstyles/
%------------------------------------------------------------------------------------------------
			\bibliographystyle{hep}
			\bibliography{biblio}
%------------------------------------------------------------------------------------------------
\end{document}